\documentclass{amsart}
\usepackage{amsmath, amssymb, amsthm,graphicx}

\usepackage[utf8]{inputenc}
\usepackage[english]{babel}
\usepackage[normalem]{ulem}
\usepackage{mathrsfs}
\usepackage{verbatim}
\usepackage{tikz}

\usetikzlibrary{fit,calc,positioning,decorations.pathreplacing,matrix}
\usetikzlibrary{arrows,decorations.markings}

\theoremstyle{plain}
\newtheorem{thm}{Theorem}[section]
\newtheorem{prp}[thm]{Proposition}
\newtheorem{cor}[thm]{Corollary}
\newtheorem{lm}[thm]{Lemma}
\newtheorem{rmq}[thm]{Remark}

\theoremstyle{definition}
\newtheorem{df}[thm]{Definition}
\newtheorem{ex}[thm]{Example}
\newtheorem{conj}[thm]{Conjecture}

\newcommand\N{\mathbb{N}}

\newcommand\D{\Delta}

\newcommand\DZ{D\!Z}

\title{Double centralizers in Artin-Tits groups}
\author{Oussama Ajbal \and Eddy Godelle}

\begin{document}

\maketitle

\begin{abstract} We prove an analogue of the Centralizer Theorem in the context of Artin-Tits groups. 

\end{abstract}

\section*{Introduction}
To obtain information on a group~$G$, a standard approach consists in considering subgroups and studying how they behave in the group.  In particular, one often consider the centralizer~$Z_G(H)$ of a subgroup~$H$ in~$G$, which is defined by  $$Z_G(H) = \{g\in G\mid gh = hg \textrm{ for all }  h\in H\}.$$ 
This general approach naturally extends to other contexts. This is the case in the study of noncommutative algebras where subgroups are replaced by subalgebras. Clearly, for an algebra~$R$ and a subalgebra~$H$, the centralizer~$Z_G(H)$ is also a subalgebra. In this framework, the subalgebra~$Z_G(Z_G(H))$, called the \emph{double centralizer} of~$H$, has been considered \cite{Far,Sim}.  For instance, a classical result \cite{Far} is the so-called \emph{Centralizer Theorem}, which claims that for a finite dimensional central simple algebra~$R$ over a field~$k$ and for a simple subalgebra~$H$, one has ~$Z_G(Z_G(H)) = H$. Various generalizations has been obtained leading to applications \cite{Tan,ChLe}. 

Regarding the result obtained in the algebra framework, and coming back to the group theory framework, one is naturally lead to consider the double-centralizer subgroup~$Z_G(Z_G(H))$ of a subgroup~$H$ in a group~$G$ and to address the question of a similar Centralizer Theorem.  Let us denote by~$\DZ_G(H)$ the double centralizer of~$H$. Obviously, when the group~$G$ has a center~$Z(G)$ that is not contained in the subgroup~$H$, the equality~$\DZ_G(H) = H$ can not hold. However, one may wonder whether the subgroup~$\DZ_G(H)$ is generated by~$Z(G)$ and~$H$. More precisely, if~$Z(G)\cap H$ is trivial, one may wonder whether~$\DZ_G(H) = Z(G)\times H$. When the center of~$G$ is trivial, we recover the property of the  Centralizer Theorem namely,~$\DZ_G(H) = H$.

As far as we know, the first Centralizer Theorem in the group theory framework has been obtained in \cite{GKLT} by considering the braid group on~$n$ strands and its standard parabolic subgroups. Our objective here is to address the more general case of  an Artin-Tits group~$G$ and a standard parabolic subgroup~$H$.  \emph{Artin-Tits} groups are those groups which possess a presentation associated with a Coxeter matrix. For a finite set~$S$, a Coxeter matrix on~$S$ is a symmetric matrix ~$(m_{s,t})_{s,t\in S}$  whom entries are either a positive integer or equal to~$\infty$, with~$m_{s,t} = 1$ if and only  if~$s = t$.  An Artin-Tits group associated with such a matrix is defined by the presentation \begin{equation}\label{presarttitsgrps}
\left\langle S\mid\underbrace{sts\ldots}_{m_{s,t}\ terms} = \underbrace{tst\ldots}_{m_{s,t}\ terms}~;\ \forall s,t\in S, s\not= t\ ; m_{s,t}\neq\infty \right\rangle. \end{equation}  For instance, If we consider~$S = \{s_1,\ldots, s_n\}$ with~$m_{s_i,s_j}= 3$ for~$|i-j| =1$ and~$m_{s_i,s_j} = 2$ otherwise, we obtain the classical presentation of the braid group~$B_{n+1}$ on~$n+1$ strings considered in \cite{GKLT}. A \emph{standard parabolic subgroup} is a subgroup generated by a subset~$X$ of~$S$. It turns out that such a subgroup is also an Artin-Tits groups in a natural way (see Proposition \ref{ThVDL}  below).  Artin-Tits groups are badly understood and most articles on the subject focus on particular subfamilies of Artin-Tits groups, such as Artin-Tits groups of spherical type, of FC type, of large type, or of 2-dimensional type.  Here again, we apply this strategy. We first consider the family of spherical type Artin-Tits groups, whom seminal example are braid groups.  We refer to the next sections for definitions. We prove:
\begin{thm} \label{theointro1} Assume~$A_S$ is a spherical type irreducible Artin-Tits group with~$S$ for standard generating set. Let~$X$ be strictly included in~$S$ and~$A_X$ be the standard parabolic subgroup of~$A$ generated by~$X$.  Denote by~$\Delta$ the Garside element of~$A_S$.
\begin{enumerate} 
\item If~$\D$ lies in~$\DZ_{A_S}(A_X)$ but not in~$Z(A_S)$, then $$\DZ_{A_S}(A_X)=A_X\times QZ(A_S)$$
\item If not, $$\DZ_{A_S}(A_X)=A_X\times Z(A_S).$$
\end{enumerate} 
\end{thm}  
In the above result we do not consider the case $X = S$. Indeed, for any group~$G$ one has~$\DZ_{G}(G) = G$. In the present article, we also consider Artin-Tits groups that are not of spherical type. We conjecture that 
\begin{conj} \label{conjintro} Assume~$A_S$ is an irreducible Artin-Tits group. Let~$A_X$ be a standard parabolic subgroup of~$A_S$ generated by a subset~$X$ of $S$. 
Assume~$A_X$ is irreducible. Let $A_T$ be the smallest standard parabolic subgroup of $A_S$ that contains $Z_{A_S}(A_X)$.
\begin{enumerate} 
\item Assume~$A_X$ is not of spherical type. Then~$\DZ_{A_S}(A_X)=Z_{A_S}(A_{T})$.
\item Assume~$A_X$ is of spherical type. 
\begin{enumerate} 
\item if $A_T$ is of spherical type, then, $$\DZ_{A_S}(A_X)=\DZ_{A_T}(A_X).$$ 
\item If~$A_T$ is not of spherical, then  $$\DZ_{A_S}(A_X)= A_X.$$
\end{enumerate}
\end{enumerate} 
\end{conj}
The centralizer of a standard parabolic subgroup is well-understood in general.   In particular, when Conjectures~1,2, 3 of \cite{God4} hold, for any given $X$, one can read on the Coxeter graph $\Gamma_S$ whether or not  the above group $A_T$ is of spherical type. This is the case for the Artin-Tits groups considered in Theorem~\ref{theointro2}. The conjecture is supported by the following result: 
\begin{thm} \label{theointro2}
\begin{enumerate}
\item Conjecture \ref{conjintro} holds for irreducible Artin-Tits groups of FC type.
\item Conjecture \ref{conjintro} holds for Artin-Tits groups of 2-dimensional type.
\item  Conjecture \ref{conjintro} holds for Artin-Tits groups of large type. 
\end{enumerate} 
\end{thm}
The reader may note that in Theorem \ref{theointro1} there is no restriction on~$A_X$, whereas in Conjecture~\ref{conjintro} we assume that~$A_X$ is irreducible. Indeed, We can extend the above conjecture to the case where $X$ not irreducible (see Conjecture~\ref{conjintrogener}) and prove that this general conjecture holds for the same Artin-Tits groups than those considered in Theorem~\ref{theointro2}. However, the statement is more technical. This is why we postpone it and restrict to the irreducible case in the introduction.     
The remainder of this article is organized as follows. In Section 2, we introduce the necessary definitions and preliminaries. Section 3 is devoted to Artin-Tits groups of  spherical type. Finally, in Section 4, we turn to the not spherical type cases.
  
\section{Preliminaries}
In this section we introduce the useful definitions and results on Artin-Tits groups that we shall need when proving our theorems.  For all this section, we consider an Artin-Tits group~$A_S$ generated by a set~$S$ and defined by Presentation~(\ref{presarttitsgrps}) given in the introduction.
\subsection{Parabolic subgroups}

As explained, the subgroups that we consider in the article are the so-called \emph{standard parabolic subgroups}, that is those subgroups that are generated by a subset of~$S$. One of the main reasons that explains why these subgroups are considered is that they are themselves Artin-Tits groups: 

\begin{prp}\cite{Vdl}\label{ThVDL}
 Let~$X$ be a subset of~$S$. Consider the Artin-Tits group~$A_X$ associated with the Coxeter matrix~$(m_{st})_{s,t\in X}$. Then 
\begin{enumerate}
\item the canonical morphism from~$A_X$ to~$A_S$ that sends~$x$ to~$x$ is into. In particular,~$A_X$ is isomorphic to, and will be identified with, the subgroup of~$A_S$ generated by~$X$.  
\item  if~$Y$ is another subset of~$S$, then we have~$A_X\cap A_Y=A_{X\cap Y}$.
\end{enumerate} 
\end{prp}
We have already defined the notion of a centralizer~$Z_{A_S}(A_X)$ of a subgroup~$A_X$. We recall that we denote the center~$Z_{A_S}(A_S)$ of~$A_S$ by~$Z(A_S)$. More generally, for a subset~$X$ of~$S$, by~$Z(A_X)$ we denote the center of the parabolic subgroup~$A_X$. Along the way, we will also need the notions of a normalizer of a subgroup and of a quasi-centralizer of a parabolic subgroups. We recall here their definitions. 
\begin{df}
 Let~$X$ be a subset of~$S$ and $A_X$ be the associated standard parabolic subgroup.
 \begin{enumerate}
 \item The \emph{normalizer} of~$A_X$ in~$A_S$, denoted  by~$N_{A_S}(A_X)$, is the subgroup of~$A_S$ defined by 
 $$N_{A_S}(A_X) = \{g\in A_S\mid g^{-1}A_Xg = A_X\}$$
 \item The \emph{quasi-centralizer} of~$A_X$ in~$A_S$, denoted  by~$QZ_{A_S}(A_X)$, is the subgroup of~$A_S$ defined by 
 $$QZ_{A_S}(A_X) = \{g\in A_S\mid g^{-1}Xg = X\}$$
 \end{enumerate}
\end{df}  
In the sequel, we will write~$QZ(A_S)$ for~$QZ_{A_S}(A_S)$.  
There is an obvious sequence of inclusion between these subgroups: $$Z_{A_S}(A_X)\subseteq QZ_{A_S}(A_X) \subseteq N_{A_S}(A_X).$$ But we can say more: 
\begin{thm}\cite{God,God_pjm,God4}\label{thm57}
Let~$A_S$ be an Artin-Tits group, and~$X$ be a subset of~$S$. 
If~$A_S$ is of spherical type or of FC type or of 2-dimensional type, then $$N_{A_S}(A_X)=QZ_{A_S}(A_X)\cdot A_X.$$
\end{thm}
This result is one of the key arguments in our proof of Theorems~\ref{theointro1} and~\ref{theointro2}. Actually, it is conjectured in~\cite{God3} that this property holds for any Artin-Tits groups.
\subsection{Families of Artin-Tits groups}
Our objective now is to introduce the various families of Artin-Tits groups that we considered in the introduction.  
\subsubsection{Irreducible Artin-Tits groups} First, we say that an Artin-Tits group is \emph{irreducible} when it is not the direct product of two of its standard parabolic subgroups. Otherwise we say that it is \emph{reducible}. Associated with the Coxeter matrix~$(m_{s,t})_{s,t\in S}$  is the Coxeter graph, which is the simple labelled graph with~$S$ as vertex set defined as it follows. There is an edge between two distinct vertices~$s$ and~$t$ when~$m_{s,t}$ is not two. The edge has label~$m_{s,t}$ when~$m_{s,t}$ is not~$3$. Therefore, the group~$A_S$ is irreducible if and only if  the Coxeter graph~$\Gamma_S$ is connected.
For instance the braid group on~$n+1$ strings is irreducible whereas the free abelian group on two generators is not. 
\subsubsection{Spherical type Artin-Tits groups} \label{secATSP}Among Artin-Tits groups, those of spherical type are the most studied and the most understood.  From Presentation~(\ref{presarttitsgrps}), we obtain the presentation of the associated Coxeter group  by adding the relations~$s^2 = 1$ for~$s$ in~$S$. The Artin-Tits group is said to be of spherical type when this associated Coxeter group is finite. For instance, braid groups are of spherical type as their associated Coxeter groups are the symmetric groups. Actually there is only a finite list of connected Coxeter graphs whom associated (irreducible) Artin-Tits groups are of spherical type (see \cite{Co},\cite{BrS}). 
\begin{center}
\begin{figure}[!h]
\begin{tikzpicture}[decoration={brace}][scale=2]
\draw[very thick,fill=black] (0,1) circle (.1cm);
\draw[very thick,fill=black] (1.3,1) circle (.1cm);
\draw[very thick,fill=black] (2.6,1) circle (.1cm);
\draw[very thick,fill=black] (3.9,1) circle (.1cm);
\draw[very thick] (0,1) -- +(3.9,0);
\draw (0.75,1.2) node {$4$};
\draw[very thick,fill=black] (5.5,1) circle (.1cm);
\draw[very thick,fill=black] (6.8,1) circle (.1cm);
\draw[very thick,fill=black] (8.1,1) circle (.1cm);
\draw[very thick,fill=black] (9.4,1) circle (.1cm);
\draw[very thick,fill=black] (10.7,1) circle (.1cm);
\draw[very thick,fill=black] (8.1,2) circle (.1cm);
\draw[very thick] (8.1,1) -- +(0,1);
\draw[very thick] (5.5,1) -- +(5.2,0);
\draw (0.75,1.2) node {$4$};
\end{tikzpicture}
\caption{Artin-Tits groups of spherical types ~$B(4)$ and~$E(6)$}
\end{figure}
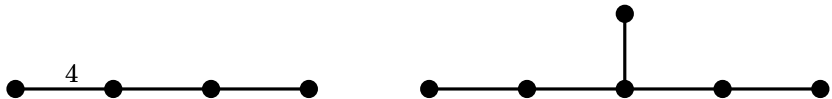
\end{center}
\subsubsection{FC type Artin-Tits groups} These Artin-Tits groups are built on those of spherical type.  An Artin-Tits group is of FC type when all its standard parabolic subgroups whom Coxeter graphs have no edge labelled with~$\infty$ are of spherical type.  In particular, all spherical type Artin-Tits groups are of FC type. Alternatively, the family of FC type Artin-Tits groups can be defined as the smallest family of groups that contains spherical type Artin-Tits groups and that is closed under amalgamation above a standard parabolic subgroup. For instance, the Artin-Tits group associated with the following Coxeter graph is of FC type. 
\begin{center}
\begin{figure}[!h]
\begin{tikzpicture}[decoration={brace}][scale=2]
\draw[very thick,fill=black] (2,1) circle (.1cm);
\draw[very thick,fill=black] (3.3,1) circle (.1cm);
\draw[very thick,fill=black] (4.6,1) circle (.1cm);
\draw[very thick,fill=black] (5.9,1) circle (.1cm);
\draw[very thick,fill=black] (2,0) circle (.1cm);
\draw[very thick,fill=black] (5.9,0) circle (.1cm);
\draw[very thick] (2,0) -- +(0,1);
\draw[very thick] (2,0) -- +(3.9,0);
\draw[very thick] (5.9,0) -- +(0,1);
\draw[very thick] (2,1) -- +(3.9,0);
\draw (2.65,1.2) node {$4$};
\draw (5.25,1.2) node {$4$};
\draw (3.95,1.2) node {$\infty$};
\end{tikzpicture}
\caption{A FC type Artin-Tits group\label{figure2}}
\end{figure}
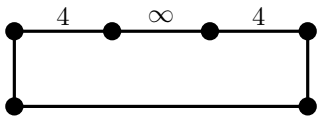
\end{center}
Indeed, the Artin-Tits group in Figure~\ref{figure2} is the amalgamation of two spherical type Artin-Tits groups of type~$B(5)$  (see \cite{Bou}) above a common standard parabolic subgroup, which is of type~$A(4)$, that is a braid group~$B_5$.  
\subsubsection{$2$-dimensional type Artin-Tits groups}  An Artin-Tits group is of 2-dimensional type when no standard parabolic subgroup generated by three, or more, generators is of spherical type. These groups has been considered, for instance, in \cite{Cha2,Che,God4}.  
\begin{center}
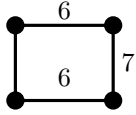
\begin{figure}[!h]
\begin{tikzpicture}[decoration={brace}]
\draw[very thick,fill=black] (3.5,1) circle (.1cm);
\draw[very thick,fill=black] (3.5,0) circle (.1cm);
\draw[very thick,fill=black] (4.8,0) circle (.1cm);
\draw[very thick,fill=black] (4.8,1) circle (.1cm);
\draw[very thick] (3.5,0) -- +(1.3,0);
\draw[very thick] (3.5,1) -- +(1.3,0);
\draw[very thick] (3.5,0) -- +(0,1);
\draw[very thick] (4.8,0) -- +(0,1);
\draw (4.15,1.2) node {$6$};
\draw (5,.5) node {$7$};
\draw (4.15,.3) node {$6$};
\end{tikzpicture}
\caption{A 2-dimensional Artin-Tits group \label{AT2D}}
\end{figure}
\end{center}
\subsubsection{Large type Artin-Tits groups}  Contained in the family of~$2$-dimensional Artin-Tits groups is the family of  Artin-Tits groups of large type. An Artin-Tits group is of large type when no ~$m_{s,t}$ is equal to~$2$.  Some~$2$-dimensional Artin-Tits groups are not of large type (see Figure~\ref{AT2D}). 
\begin{center}
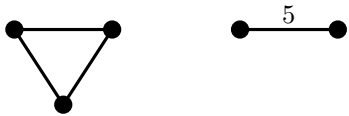
\begin{figure}[!h]
\begin{tikzpicture}[decoration={brace}][scale=2]
\draw[very thick,fill=black] (2,1) circle (.1cm);
\draw[very thick,fill=black] (3.3,1) circle (.1cm);
\draw[very thick,fill=black] (2.65,0) circle (.1cm);
\draw[very thick] (2,1) -- +(1.3,0);
\draw[very thick] (2,1) -- +(0.65,-1);
\draw[very thick] (3.3,1) -- +(-0.65,-1);
\draw[very thick,fill=black] (5,1) circle (.1cm);
\draw[very thick,fill=black] (6.3,1) circle (.1cm);
\draw[very thick] (5,1) -- +(1.3,0);
\draw (5.65,1.2) node {$5$};
\end{tikzpicture}
\caption{Artin-Tits groups of large types~$\tilde{A}(2)$ and~$I(5)$. \label{ATLT}}
\end{figure}
\end{center}
\subsection{Artin-Tits monoids}
As explained above, one of the main ingredients  in our proof is Theorem \ref{thm57}. Another one is the positive monoid of an Artin-Tits monoid that allows to apply Garside theory. Here, we introduce only the results that we will need and refer to \cite{DDGKM} for more details on this theory. 
We recall that we fix an Artin-Tits group~$A_S$ generated by a set~$S$ and defined by Presentation~(\ref{presarttitsgrps}). 
\begin{df} The Artin-Tits monoid~$A_S^{+}$ associated with~$A_S$ is the submonoid of~$A_S$ generated by~$S$. An element of~$A_S$ that belongs to~$A_S^+$ is called a positive element. Its inverse is called a negative element.
\end{df}
We gather in the following proposition several properties of Artin-Tits monoids that we will need in the sequel. 
\begin{prp}\label{reststdAT}
\begin{enumerate}
\item \cite{Par4}  Considered as a presentation of monoid, Presentation~(\ref{presarttitsgrps}) is a presentation of the monoid~$A_S^+$. 
\item When~$A_S$ is of spherical type, then
\begin{enumerate}
\item \cite{BrS,Cha1,DDGKM}  the monoid~$A^+_S$ is a Garside monoid. In particular, every element~$g$  in~$A_S$ can be decomposed in a unique way as~$g=a^{-1}b$, with~$a,b$ positive, so that~$a$ and~$b$ have no nontrivial common left-divisors in~$A_S^+$. Furthermore, if~$c\in A_S^+$ is such that~$cg\in A_S^+$, then~$a$ right-divides~$c$ in~$A^+_S$.
\item \cite{BrS,Del} There is a positive element~$\Delta$ that belongs to~$QZ(A_S)$ so that every element~$g$  in~$A_S$, can be decomposed as ~$g=a\Delta^{-n}$ with~$a$ positive and~$n\geq 0$.  Moreover,~$\Delta^2$ belongs to~$Z(A_S)$.
\item \cite{BrS,Del} When, moreover,~$A_S$ is irreducible then ~$QZ(A_S)$ is an infinite cyclic group generated by~$\Delta$. The group~$Z(A_S)$ is infinite cyclic generated by~$\Delta$ or by~$\Delta^2$.  
\end{enumerate}
\end{enumerate}
\end{prp}
The decomposition~$g=a^{-1}b$ in Point~(ii)(a) is called the {\em Charney's (left) orthogonal splitting} of~$g$. The {\em Charney's right orthogonal splitting}~$g=ab^{-1}$ is defined in a similar way. 

In the sequel, we denote by~$\tau:S\to S$  the permutation of~$S$ defined by~$\D s=\tau(s)\D$ for all~$s$ in~$S$. As explained above,~$\tau$ is either the identity or an involution. In particular, we have also~$s\D=\D\tau(s)$ for all~$s$ in~$S$. Moreover, for~$a,b$ in~$A_S^+$, we write~$a\preceq b$ if~$a$ left-divides~$b$ in~$A_S^+$, that is if there exists~$c$ in~$A_S^+$ so that~$b = ac$. Similarly, we write~$b\succeq a$ if~$a$ right-divides~$b$ in~$A_S^+$.

\section{Spherical type Artin-Tits groups}

In this Section we focus on spherical type Artin-Tits groups and prove Theorem~\ref{theointro1}.
\subsection{Artin-Tits groups of type~$E(6)$ and~$D(2k+1)$}
In Theorems~\ref{theointro1} the description of~$\DZ_{A_S}(A_T)$ depends on a technical condition. 
Here we investigate this condition and characterize irreducible Coxeter graphs for which this condition is satisfied.  
\begin{prp}\label{condtech} Assume~$A_S$ is an irreducible spherical type Artin-Tits group.  Let~$X$ be a proper subset of~$S$. Then, ~$\D$ does not belong to~$Z(A_S)$ but lies in~$\DZ_{A_S}(A_X)$ if and if :
\begin{enumerate}
\item[(a)] either~$\Gamma_S$ is of type~$D(2k+1)$ and~$X\supseteq\{s_2,s_{2'},s_3\}$ (see  Figure~\ref{figuretypeD}).
\item[(b)] or~$\Gamma_S$ is of type~$E_6$ and~$X=\{s_2,\dots,s_6\}$ (see  Figure~\ref{figuretypeE6}).
\end{enumerate}
\begin{center}
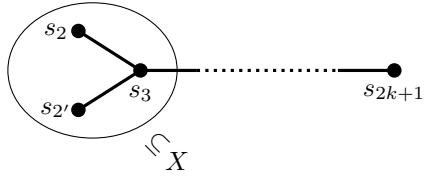
\begin{figure}[!h]\label{figuretypeD}
\begin{tikzpicture}[scale=0.75]
\draw[very thick,fill=black] (3.5,-3.5) circle (.1cm);
\draw[very thick,fill=black] (8,-3.5) circle (.1cm);
\draw[very thick,fill=black] (2.4,-2.8) circle (.1cm);
\draw[very thick,fill=black] (2.4,-4.2) circle (.1cm);
\draw[very thick] (3.5,-3.5) -- +(1,0);
\draw[very thick] (2.4,-2.8) -- +(1.1,-.7);
\draw[very thick] (2.4,-4.2) -- +(1.1,.7);
\draw[dotted,very thick] (4.5,-3.5) -- +(2.5,0);
\draw[very thick] (7,-3.5) -- +(1,0);
\draw (2,-2.8) node {$s_2$};
\draw (2,-4.2) node {$s_{2'}$};
\draw (3.5,-3.9) node {$s_3$};
\draw (8,-3.9) node {$s_{2k+1}$};
\draw (4.15,-5.1) node {$X$};
\node at (3.8,-4.8) [rotate=315] {$\subseteq$};
\draw (2.65,-3.5) ellipse (1.5cm and 1.2cm);
\end{tikzpicture}
\caption{$\Gamma_S$ of type~$D(2k+1)$ and~$X\supseteq\{s_2,s_{2'},s_3\}$ \label{figuretypeD}} 
\end{figure}
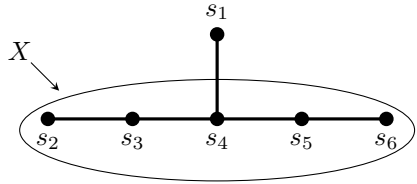
\begin{figure}[!h]
\begin{tikzpicture}[scale=0.75]
\draw[very thick,fill=black] (2,-6.7) circle (.1cm);
\draw[very thick,fill=black] (3.5,-6.7) circle (.1cm);
\draw[very thick,fill=black] (5,-6.7) circle (.1cm);
\draw[very thick,fill=black] (5,-5.2) circle (.1cm);
\draw[very thick,fill=black] (6.5,-6.7) circle (.1cm);
\draw[very thick,fill=black] (8,-6.7) circle (.1cm);
\draw[very thick] (2,-6.7) -- +(6,0);
\draw[very thick] (5,-5.2) -- +(0,-1.5);
\draw (2,-7.1) node {$s_2$};
\draw (3.5,-7.1) node {$s_3$};
\draw (5,-7.1) node {$s_4$};
\draw (6.5,-7.1) node {$s_5$};
\draw (8,-7.1) node {$s_6$};
\draw (5,-4.8) node {$s_1$};
\draw (1.5,-5.5) node {$X$};
\draw (5,-6.9) ellipse (3.5cm and .9cm);
\draw[-stealth] (1.7,-5.7) -- +(.5,-.5);
\end{tikzpicture}
\caption{$\Gamma_S$ of type~$E_6$ and~$X=\{s_2,\dots,s_6\}$ \label{figuretypeE6}} 
\end{figure}
\end{center}
\end{prp}

When proving Proposition~\ref{condtech}, we will need the following lemma.

\begin{lm}\label{lm518}
Assume~$A_S$ is an irreducible spherical type Artin-Tits group.  Let~$X$ be a proper subset of~$S$.  Assume the permutation~$\tau$ is not the identity on~$S$ and~$\D$ lies in~$\DZ_{A_S}(A_X)$ then:
\begin{enumerate}
\item~$\tau$ is the identity on~$S\setminus X$, that is~$\D$ lies in~$Z_{A_S}(A_{S\setminus X})$.
\item~$\tau$ is not the identity on~$X$, that is~$\D$ does not lie in~$Z_{A_S}(A_X)$.
\item~$\D$ stabilizes the indecomposable components of~$X$.

\end{enumerate}
\end{lm}
\begin{proof}
(i) Let~$s\in S\setminus X$. Set~$Y=X\cup\{s\}$. The elements~$\D_X^2,\D_Y^2$ lie in~$Z(A_X)$ and~$Z(A_Y)$, respectively. So, they both belong to~$Z_{A_S}(A_X)$ and, therefore, commute with~$\D$. Since ~$\D\D_X=\D_{\tau(X)}\D$ and~$\D\D_Y=\D_{\tau(Y)}\D$, we deduce that~$\tau(X) = X$ and~$\tau(Y) = Y$. Using that~$Y=X\cup\{s\}$, we concluded that ~$\D s=s\D$. Thus, ~$\D$ lies in~$Z_{A_S}(A_{S\setminus X})$ and (i) holds. Since ~$\tau$ is not the identity on~$S$, (i) implies (ii).  Finally,  Let~$X_1$ be an indecomposable component of~$X$. We have~$\D\D_{X_1}=\D_{\tau(X_1)}\D$. Then,~$\D_{X_1}^2$ lies in~$Z(A_X)$ and, therefore, in~$Z_{A_S}(A_X)$. Hence~$\D\D_{X_1}^2=\D_{X_1}^2\D$ and~$X_1 = \tau(X_1)$, that is~$\D X_1=X_1\D$.
\end{proof}

\begin{proof}[Proof of Proposition \ref{condtech}]
Assume the element~$\D$ does not belong to the center~$Z(A_S)$ but lies in~$\DZ_{A_S}(A_X)$. So, assertions (i)(ii) and (iii) in Lemma~\ref{lm518} hold. In particular, the permutation~$\tau$ is not the identity map on~$S$. Using the classification of irreducible Artin-Tits groups~\cite{Bou} and well-known results on~$\D$ \cite{BrS,Del}, we deduce that the type of~$\Gamma_S$ is one of the following types: 
\begin{enumerate}
\item[$\bullet$]~$A(k)$ with~$k\geq 2$, 
\item[$\bullet$]~$D(2k+1)$ with~$k\geq 1$,  
\item[$\bullet$]~$E_6$, or  
\item[$\bullet$]~$I_2(2p+1)$ with ~$p\geq 1$.
\end{enumerate}

By Lemma \ref{lm518}(i), the permutation~$\tau$ fixes each element of~$S\setminus X$. This imposes
 that~$\Gamma_S$ cannot be of type~$I_2(2p+1)$, as~$X$ is proper in~$S$ and~$\Delta$ permutes the two elements of~$S$.  If~$\Gamma_S$ is of type~$A(k)$ with~$k\geq 2$, (so~$A_S$ is  the braid group~$B_{k+1}$) then the unique element of~$S$ fixed by~$\tau$ is~$s_{\frac{k+1}{2}}$. This imposes~$S\setminus X=\{s_{\frac{k+1}{2}}\}$.  and~$\D$ does not fix the two indecomposable components {$\{s_1,\dots,s_{\frac{k-1}{2}}\}$ and~$\{s_{\frac{k+3}{2}},\dots,s_k\}$} of~$X$, a contradiction with Lemma \ref{lm518}(iii). So~$\Gamma_S$ is not of type~$A(k)$. If~$\Gamma_S$ is of type~$D(2k+1)$, then~$\tau$ switches~$s_2$ and~$s_{2'}$. Therefore,~$s_2$ and $s_{2'}$ have to lie in~$X$. Moreover,~$s_2$ does not commute with~$\Delta$, so it cannot belong to~$Z_{A_S}(A_X)$. This imposes that~$s_3$ has to belong to~$X$. Hence, ~$\{s_2,s_{2'},s_3\}$ is included in~$X$ and we have case (a) of the proposition. Similarly, if~$\Gamma_S$ is of type~$E_6$.  The elements~$s_2,s_3,s_5,s_6$ are not fix by~$\tau$, so they have to belong to~$X$. Applying Lemma \ref{lm518}(iii), we deduce that~$s_4$ has to lie in~$X$ too. Since~$X$ is not~$S$, it is equal to~$\{s_2,\dots,s_6\}$ and we have case (b) of the proposition. Conversely,  in case (a) and (b), one can verify that~$\D$ does not belong to~$Z(A_S)$ but lies on~$\DZ_{A_S}(A_X)$
\end{proof}
\subsection{Ribbons}\label{sectribbon}
  The notion of ribbon introduced in \cite{FRZ} for the case of braid groups, and then generalized in \cite{Par1,Got}, will be crucial to us in order to calculate the double-centralizer of a parabolic subgroup.  Here we recall its definition and gather some properties that we shall need. In particular, we only consider spherical type Artin-Tits groups. We refer to above references and to \cite{DDGKM} for more details. 
 Given an Artin-Tits presentation~(\ref{presarttitsgrps}), let us first introduce two notations: for  a subset~$X$ of~$S$, we set $$X^\bot=\{s\in S\setminus X\mid\forall t\in X,m_{ts}=2\}$$ and $$\partial X=\{s\in S\setminus X\mid\exists t\in X,m_{ts}>2\}.$$
\begin{center}
\begin{figure}[!h]
\begin{tikzpicture}[scale=0.75]
\draw[very thick,fill=black] (2,-13.1) circle (.1cm);
\draw[very thick,fill=black] (3.5,-13.1) circle (.1cm);
\draw[very thick,fill=black] (5,-13.1) circle (.1cm);
\draw[very thick,fill=black] (5,-11.6) circle (.1cm);
\draw[very thick,fill=black] (6.5,-13.1) circle (.1cm);
\draw[very thick,fill=black] (8,-13.1) circle (.1cm);
\draw[very thick,fill=black] (9.5,-13.1) circle (.1cm);
\draw[very thick,fill=black] (11,-13.1) circle (.1cm);
\draw[very thick] (2,-13.1) -- +(9,0);
\draw[very thick] (5,-11.6) -- +(0,-1.5);
\draw (2,-13.5) node {$s_2$};
\draw (3.5,-13.5) node {$s_3$};
\draw (5,-13.5) node {$s_4$};
\draw (6.5,-12.7) node {$s_5$};
\draw (8,-13.5) node {$s_6$};
\draw (9.5,-13.5) node {$s_7$};
\draw (11,-13.5) node {$s_8$};
\draw (5,-11.2) node {$s_1$};
\draw (11.8,-12.2) node {$X^\bot$};
\draw (9.5,-13.3) ellipse (2cm and .8cm);
\draw[-stealth] (11.4,-12.4) -- +(-.45,-.25);
\draw (1.2,-12.2) node {$X$};
\draw (3.5,-13.3) ellipse (2cm and .8cm);
\draw[-stealth] (1.45,-12.4) -- +(.45,-.25);
\draw (7.1,-11.2) node {$\partial X$};
\draw [rotate around={135:(5.75,-12.15)}] (5.75,-12.15) ellipse (1.7cm and .6cm);
\draw[-stealth] (6.7,-11.4) -- +(-.45,-.25);
\end{tikzpicture}
\caption{Example :~$\partial(X)$ and~$X^\bot$}\label{exampledeltaetperp}
\end{figure}
\end{center}
\begin{df}\label{defribb}
   \begin{enumerate}
\item Let~$t$ belong to~$S$ and~$X$ be included in~$S$.  Denote by~$X(t)$ the indecomposable component of~$X\cup\{t\}$ containing~$t$. If~$t$ lies in~$X$, we set ~$d_{X,t}=\D_{X(t)}$; otherwise, we set  $$d_{X,t} = \D_{X}\D_{X-\{t\}}^{-1},$$ that is~$d_{X,t} = \D_{X(t)}\D_{X(t)-\{t\}}^{-1}$. In both cases, there exists~$Y\subseteq X\cup\{t\}$ and~$t'\in X(t)$ so that~$Y d_{X,t} = d_{X,t} X$ with~$Y\cup\{t'\} = X\cup\{t\}$ and~$Y(t')= X(t)$. The element~$d_{X,t}$ is called a {\em positive elementary~$Y$-ribbon-$X$}.
\item For~$X,Y\subseteq S$, we say that~$g\in A_S^+$ is a positive~$Y$-ribbon-$X$ if ~$Yg = gX$. \end{enumerate}\end{df}
For instance,  considering the example in Figure~\ref{exampledeltaetperp}, ~$d_{X,s_5}  = s_2s_3s_4s_5$,~$d_{X,t} = t$  for~$t$ in~$X^\perp$ and~$\Delta$ is a positive~$X$-ribbon-$X$. 

The connection between positive ribbons and elementary ones appears in the following result    
\begin{prp} Assume~$A_S$ is a spherical type Artin-Tits group and~$g$ lie in~$A_S^+$.
~$g$ is a positive~$Y$-ribbon-$X$ if and only if ~$g=g_n\cdots g_1$ where each~$g_i$ is a positive elementary~$X_i$-ribbon-$X_{i-1}$, with~$X_0=X$ and~$X_n=Y$. 
\end{prp}
\begin{prp}\label{lm58}
Assume~$A_S$ is a spherical type Artin-Tits group. Let~$X$ be included in~$S$ and~$u$ be included in $A_S^+$.  Let~$\varepsilon\in\{1,2\}$ be such that~$\D_X^\varepsilon$ lies in~$Z(A_X)$.
\begin{enumerate} 
\item Assume $u$ is a positive $Y$-ribbon-$X$ for some $Y\subseteq S$.
\begin{enumerate}
\item  $\D_Yu=u\D_X$.
\item Assume~$t$ belongs to~$S$. Then,~$u\succeq t~\Leftrightarrow~u\succeq d_{X,t}$.
\end{enumerate}
\item If~$u\D_X^\varepsilon\succeq u$, then there exists $Y\subseteq S$ such that $u\D_X^\varepsilon u^{-1}=\D_{Y}^\varepsilon$, $uA_X u^{-1}=A_{Y}$  and $\Gamma_X\sim\Gamma_{Y}$. Moreover, if $u$ is reduced-$X$, then $u$ is a positive $Y$-ribbon-$X$, that is $Yu = uX$. 
\end{enumerate}
\end{prp}
The above results are not all explicitly stated in \cite{Par1,Got} but are well-known from specialists. The second part of (ii) is stated in \cite[Lemma 2.2]{Got} and the first part follows (see also \cite[Lemma 5.6]{Par1}). Point~(i) is proved in the proof of \cite[Lemma~2.2]{Got} (see \cite[Lemma 5.6]{Par1} for details). For point (i)(b), see also \cite[Example 3.14]{God6}. \\    
The \emph{support} of a word on~$D$ is the set of letters that are involved in this word. It follows from the presentation of~$A^+_S$ that two representing words of the same element in~$A_S^+$  have the same support. So the {\em support} of an element of~$A_S^+$ is well-defined. In the sequel, by~$Supp(g)$ we denote the support of an element~$g$ in~$A_S^+$.\\
\begin{lm}\label{lm59} Assume~$A_S$ is a spherical type Artin-Tits group. 
Let~$X\subsetneq S$ be such that~$\Gamma_X$ is connected, and~$t\in\partial X$. Then $$Supp(d_{X,t})=X\cup\{t\}.$$
\end{lm}
\begin{proof} By assumption~$t$ is not in~$S$, so~$d_{X,t} =  \D_{X}\D_{X-\{t\}}^{-1}$ and 
$Supp(d_{X,t})$ is included in~$X\cup\{t\}$. Let us show the converse inclusion. Now,  by Proposition~\ref{lm58} (i), we have~$d_{X,t} = v_0s_0$ for some~$v_0$ in~$A_S^+$ and~$s_0=t$ belongs to the support of~$d_{X,t}$.  Let~$s$ be in~$X$. Since~$X$ is connected and~$t$ belongs to~$\partial X$, there exists a finite sequence  ~$s_1,\dots,s_n$ of~$X$ such that~$s_n = s$ and for all~$i\geq0$, we have~$m_{s_i,s_{i+1}}\neq 2$. We assume the sequence is chosen so that~$n$ is minimal. Assume~$d_{X,t} = v_i s_i\cdots s_0$ for some~$0 \leq i<n$ with~$v_i$ in~$A_S^+$.  Since~$Yd_{X,t} = d_{X,t}X$ for some~$Y\subseteq X\cup\{t\}$, we can write ~$v_is_i\cdots s_0s_{i+1} =  s'_{i+1}v_is_i\cdots s_0$ for some~$s'_{i+1}$ in~$X\cup\{t\}$. By minimality of~$n$,~$m_{s_js_{i+1}}=2$ for any~$j<i$. So~$v_is_is_{i+1}s_{i-1}\cdots s_0 = s'_{i+1}v_is_i\cdots s_0$ and~$v_is_is_{i+1} =  s'_{i+1}v_is_i$. This imposes that~$v_is_is_{i+1} =  s'_{i+1}v_is_i = v'\underbrace{\cdots s_{i+1}s_is_{i+1}}_{m\ terms}$ with~$m = m_{s_i,s_{i+1}}$ and~$v'$ in~$A_S^+$ (see  \cite{Del,BrS}). This imposes in turn that we can write~$v_i = v_{i+1}s_{i+1}$ and~$d_{X,t} = v_i s_i\cdots s_0$.  Then, we obtain step-by-step that~$d_{X,t}$ can be decomposed as~$v_n s_n\cdots s_0$. Hence~$s$ belongs to the support of~$d_{X,t}$ for any~$s$ in~$X$.  So the converse inclusion holds. 
\end{proof}
\begin{lm}\label{lm512}
Let~$u\in A_S^+$ and~$s\in S$. Denote by~$u_2^{-1}v_1$ the left orthogonal splitting of the element~$u^{-1}su$. There exists~$u_1$ in~$A^+_S$  and~$s_1$ in~$S$ so  that~$u = u_1u_2$, $v_1 = s_1u_2$. Moreover, $u_1$ is a positive~$s$-ribbon-$s_1$ .
\end{lm}
\begin{proof}
 By \cite[Theorem~1]{GKT} there exists~$u_1$ in~$A^+_S$ so that~$u = u_1u_2$ and~$v_1 = s_1u_2$ for some~$s_1$ in~$S$. Moreover, applying~\cite[Lemma~2.3]{GKT}, a straightforward induction on the length of $u$ proves that~$u_1$ is a positive ~$s$-ribbon-$s_1$. 
\end{proof}
In the sequel, we say that an element of~$A_S^+$ is a positive ribbon-$X$ when it is a positive~$Y$-ribbon-$X$, for some~$Y$. Similarly we say that an element is a positive $Y$-ribbon when it is a positive~$Y$-ribbon-$X$.

\subsection{The proof of Theorem~\ref{theointro1}.}
In  this section we prove Theorem~\ref{theointro1}. The proof need two preliminary results, namely Lemma~\ref{lm517} and Proposition \ref{prp53}, which is the main argument. The latter is proved here; the proof of the former is postponed  to the next section.

\begin{lm}\label{lm517}
Under the assumptions of Proposition \ref{prp53}, we have $b\succeq s$ for all $s\in S\setminus X$.
\end{lm}

\begin{prp}\label{prp53}
Let~$A_S$ be an irreducible Artin-Tits group of spherical type. Let~$X\subsetneq S$. Let~$b$ be in~$A_S^+\setminus\{1\}$ a positive ribbon-$(X\cup X^\bot)$ that is reduced-$X$. Suppose further that for all~$Y\subseteq S$  containing~$X$, and~$\varepsilon(Y)\in\{1,2\}$ minimal such that~$\D_Y^{\varepsilon(Y)}\in Z_{A_S}(A_X)$, we have~$b\D_Y^{\varepsilon(Y)}\succeq b$. Then there exists~$n\in\N^*$ so that $$b=\D^n\D_X^{-n}$$ 
\end{prp}
Note that~$\Delta_X$ right-divides~$\Delta$ in~$A_S^+$ and~$\Delta \Delta_X = \Delta_{\tau(X)}\Delta$ by Proposition~\ref{reststdAT}. So for any positive integer~$n$  the element~$\D^n\D_X^{-n}$ belongs to~$A_S^+$. 
\begin{proof}[Proof of Proposition \ref{prp53}]
We have~$\D_X\succeq s$ for all~$s\in X$ and,  by Lemma \ref{lm517}, we have~$b\succeq s$ for all~$s\in S\setminus X$. Since, by assumption,~$b\D_X\succeq b$, we get that~$b\D_X\succeq s$ for all~$s\in S$. Thus~$b\D_X\succeq\D$ and, therefore,~$b\succeq\D\D_X^{-1}$ in~$A_S^+$. Let~$k\in\N^*$ be maximal such that~$b\succeq\D^k\D_X^{-k}$. Write~$b=d\D^k\D_X^{-k}$ with~$d\in A_S^+$. We show that~$d=1$.  This will prove the proposition. Since~$\D X=\tau(X)\D$,  the element~$\D^k$ is a positive~$\tau^k(X)$-ribbon-$X$ and~$\D^kX=\tau^k(X) \D^k$. Therefore by Proposition~\ref{lm58}, we have~$\D^k\D_X=\D_{\tau^k(X)}\D^k$. For the remaining of the proof, for~$Z\subseteq S$, we set~$Z_k = \tau^k(Z)$. Moreover,~$\D_X$ is a positive~$X$-ribbon-$X$. Then~$\D^k\D_X^{-k}$ is a positive~$X_k$-ribbon-$X$. For the remaining of the proof, when~$s$ lies in~$X_k$, we denote by~$s_X$ the element of~$X$ so that~$s\D^k\D_X^{-k} = \D^k\D_X^{-k}s_X$. 

Assume there exists~$s$ in~$X_k$ so that~$d = u s$ with~$u$ in~$A_S^+$.  Then we have~$b= us\D^k\D_X^{-k}$.  We get~$b = u\D^k\D_X^{-k}s_X$. But this is not possible, since~$b$ is reduced-$X$. Hence,~$d$ is reduced-$X_k$.  We now prove that~$d$ is a positive ribbon-$(X_k\cup X_k^\bot)$. Let~$s$ lie in~$X_k$. We have~$s\D^k\D_X^{-k}=\D^k\D_X^{-k}s_X$. By assumption,~$b$ is a positive ribbon-$X$, therefore there exists~$s'$ in~$S$ so that~$bs_X = s' b$. Hence we get~$ds\D^k\D_X^{-k} = d\D^k\D_X^{-k}s_X =  s'd\D^k\D_X^{-k}$, and therefore~$ds = s'd$.  As this so for every element of~$X_k$, we deduce that~$d$ is a positive ribbon-$X_k$.  Let~$s$ lie~$X_k^\bot$. For every~$t$ in~$X$, we have~$\tau^k(t)$ lies in~$X_k$ and, therefore,~$m_{\tau^k(t),s} = 2$. But the involution~$\tau$ induces an automorphism of the Coxeter graph associated with the presentation of~$A_S$. It follows that for every~$t$ in~$X$, we have~$m_{t,\tau^k(s)}=m_{\tau^k(t),s} = 2$. Hence, ~$\tau^k(s)$ belongs to~$X^\bot$. But~$b$ is a positive ribbon-$X^\bot$, then~$b\tau^k(s) =  s'b$ for some~$s'\in S$. Hence we get~$ds\D^k\D_X^{-k} = d\D^k\D_X^{-k}\tau^k(s) =  s'd\D^k\D_X^{-k}$, and therefore~$ds = s'd$. As this so for every element of~$X^\bot_k$, we deduce that ~$d$ is a positive ribbon-$X_k^\bot$. Gathering the two results we get that~$d$ is a positive ribbon-$(X_k\cup X_k^\bot)$.  

Let~$Y\subseteq S$ containing~$X_k$ and consider~$\eta(Y)$ be positive and minimal such that~$\D_Y^{\eta(Y)}$ belongs to~$Z_{A_S}(A_{X_k})$. The involution~$\tau^k$ exchanges~$X$ and~$X_k$ and~exchanges $Y$ and~$Y_k$. It follows, Firstly, that the inclusion,~$X_k\subseteq  Y$ implies  the inclusion~$X\subseteq Y_k$ and, Secondly,  
 that~$\tau^k$ send~$A_{X_k}$ and~$\D_Y$ to~$A_X$ and~$\D_{Y_k}$, respectively, with~$\eta(Y) = \varepsilon(Y_k)$. Thus,~$\D_{Y_k}^{\eta(Y)}$ belongs to~$Z_{A_S}(A_X)$ with~$\eta(Y) = \varepsilon(Y_k)$.  Then, by assumption, we have~$b\D_{Y_k}^{\eta(Y)}=  ub$, for some~$u$ in~$A_S^+$. Since ~$b\D_{Y_k}^{\eta(Y)} = d\D^k\D_X^{-k} \D_{Y_k}^{\eta(Y)} = d\D^k \D_{Y_k}^{\eta(Y)} \D_X^{-k} = d\D_{Y}^{\eta(Y)}\D^k \D_X^{-k}$ and~$ub = u d\D^k\D_X^{-k}$ we obtain that~$d\D_{Y}^{\eta(Y)} = ud$. As a consequence, replacing~$b$ and~$X$  by~$d$ and~$X_k$, respectively, we can repeat the beginning of the argument and deduce that ~$d = d_1\D\D_{X_k}^{-1}$ for some~$d_1$ in~$A_S^+$. But this lead to a contradiction to the maximality of~$k$, since  we get~$b = d\D^k\D_X^{-k} =   d_1\D\D_{X_k}^{-1}\D^k\D_X^{-k}=\D^{k+1}\D_X^{-(k+1)}$. Hence~$d = 1$ and~$b = \D^k\D_X^{-k}$.
 \end{proof}
We turn now to the proof of Theorem~\ref{theointro1}.

\begin{proof}[Proof of Theorem \ref{theointro1}]
Let~$u$ lie in $\DZ_{A_S}(A_X)$. We have~$Z_{A_X}(A_X)\subseteq Z_{A_S}(A_X)$, then~$\DZ_{A_S}(A_X)\subseteq Z_{A_S}(Z_{A_X}(A_X))$ and $u$ belongs to $Z_{A_S}(Z_{A_X}(A_X))$. Thanks to Theorem \ref{thm57}, we can write~$u=y\cdot z$, with~$yX=Xy$ and~$z\in A_X$. Write (see Proposition~\ref{reststdAT})~$y=\D^{-2m}h$ with $h$ in $A_S^+$, and decompose $h$ as~$h=abc$, with~$a,c\in A_X^+$ and~$b$ being~$X$-reduced-$X$. Since~$yX=Xy$ and $\D^2$ is in $Z(A_S)$, we have~$hX=Xh$ and so~$h\D_X=\D_Xh$. Using that $h = abc$ with~$a,c$ in $A_X^+$ and that~$\D^2_X$ lie in $Z(A_X)$, we deduce that~$b\D_X^2=\D_X^2b$ . The element~$b$ is reduced-$X$, then by Proposition~\ref{lm58}, we have~$bX=Xb$. It follows there exists~$z'\in A_X$ such that~$bcz=z'b$. Set~$x=az'$. Then, $x$  belongs to $A_X$ and~$u=\D^{-2m}\cdot x\cdot b$. Suppose~$b\neq 1$. By definition~$X^\bot\subseteq Z_{A_S}(A_X)$. Therefore, for all~$s\in X^\bot$ we have~$us=su$ and $s\D^{-2m}\cdot x\cdot b = \D^{-2m}\cdot x\cdot s b$.  By cancellation, we obtain~$bs=sb$ for all~$s\in X^\bot$. So, $b$ is a positive ribbon-$(X\cup X^\perp)$. Now, let $Y$ be included in $S$ and containing $X$. Set $\varepsilon(Y)\in\{1,2\}$ be minimal such that~$\D_Y^{\varepsilon(Y)}$ lies in $Z_{A_S}(A_X)$. Then,~$u\D_Y^{\varepsilon(Y)}=\D_Y^{\varepsilon(Y)}u$ and, as before, we get~$b\D_Y^{\varepsilon(Y)}=\D_Y^{\varepsilon(Y)} b$. As a consequence we have~$b\D_Y^{\varepsilon(Y)}\succeq b$. By Proposition \ref{prp53}, we deduce there exists~$n\in\N^*$ so that~$b=\D^n\D_X^{-n}$. Thus, we get~$u=\D^{-2m}x\D^n\D_X^{-n}$.

Assume, First, that~$\D\in \DZ_{A_S}(A_X)$ and~$\D\notin Z(A_S)$. Then, by Lemma \ref{lm518}, we have~$\D X=X\D$ and $\tau^n(x)$ belongs to $A_X$. Therefore,~$u=\D^{-2m+n}\cdot \tau^n(x)\D_X^{-n}$ and  $u$ belongs to $QZ(A_S)\cdot A_X$. So~$\DZ_{A_S}(A_X)$ is included in~$QZ(A_S)\cdot A_X$. Conversely, The assumption that~$\D$ lies in $\DZ_{A_S}(A_X)$ imposes the inclusion~$QZ(A_S)\cdot A_X\subseteq \DZ_{A_S}(A_X)$. Therefore, the latter inclusion is actually an equality. Moreover we have~$QZ(A_S)\cdot A_X=A_X\cdot QZ(A_S)$, since $\D$ belongs  to $QZ_{A_S}(A_X)$ by the above argument.

Assume, Secondly that~$\D\notin \DZ_{A_S}(A_X)$ or~$\D\in Z(A_S)$. First, the inclusion~$Z(A_S)\cdot A_X\subseteq \DZ_{A_S}(A_X)$ holds in any case. If~$\D\in Z(A_S)$ or $n$ is even, then~$u$, that is $\D^{-2m+n}x\D_X^{-n}$,  lies in  $Z(A_S)\cdot A_X$ and so the other inclusion~$\DZ_{A_S}(A_X)\subseteq Z(A_S)\cdot A_X$ holds too. Assume finally~$\D\notin \DZ_{A_S}(A_X)$. Since~$u$ lies in $\DZ_{A_S}(A_X)$, for every $w$ in~$Z_{A_S}(A_X)$  we have $wu=uw$ and therefore~$\D^{-2m}xw\D^n\D_X^{-n}=\D^{-2m}x\D^nw\D_X^{-n}$. This, in turn, imposes~$\D^nw=w\D^n$ for every $w$ in~$Z_{A_S}(A_X)$. In other words~$\D^n$ lies in $\DZ_{A_S}(A_X)$ too. Since  $\D^2$ lies in $\DZ_{A_S}(A_X)$ but not $\D$, we deduce that $n$ has to be even, and conclude by the above argument that $Z(A_S)\cdot A_X = \DZ_{A_S}(A_X)$ .   

Finally, we note that $A_X \cap QZ(A_S) = A_X \cap Z(A_S) = \{1\}$. Indeed, $X\neq S$ and $Supp(\D) = S$. Therefore, $\D^m$  does not belong to $A_X$ except if $m = 0$. Hence, we hace~$A_X\cdot QZ(A_S) = A_X\times QZ(A_S)$ and~$A_X\cdot Z(A_S) = A_X\times Z(A_S)$.
\end{proof}
\subsection{The proof of Lemma \ref{lm517}}
Here we focus on the proof of Lemma \ref{lm517}, that was postponed in the previous section. It is technical and, to help the reader, we decompose in 3 steps, namely Lemma~\ref{lm514}, Lemma~\ref{lm516} and the final argument.    
\begin{lm}\label{lm514}
Under the assumptions of Proposition \ref{prp53}, if~$t\in\partial X$ then $$b\nsucceq t\ \Leftrightarrow\ bt = t'b \textrm{ for some }t'\in S.$$
\end{lm}
\begin{proof}
Assume~$b\nsucceq t$. Set~$Y=X\cup\{t\}$. Under the assumptions of Proposition~\ref{prp53}, we have~$b\D_Y^{\varepsilon(Y)}\succeq b$. By Proposition~\ref{lm58}, we deduce that~$b\D_Y^{\varepsilon(Y)} b^{-1}=\D_{Y'}^{\varepsilon(Y)}$ and~$bA_Y b^{-1}=A_{Y'}$ for some subset $Y'$ of $S$. On the other hand, $b$ is a positive $X'$-ribbon-$X$ for some subset $X'$ of $S$. It follows that~$X'$ is included in~$A_{Y'}$ and, therefore, in~$Y'$. Now, the sets~$X'$ and~$Y'$ have the same cardinality as~$X$ and~$Y$, respectively. Then there exists~$t'$ in~$Y'$ so that~$Y'=X'\cup\{t'\}$. We are going to prove that~$bt = t'b$.  By Lemma~\ref{lm512}, we can decompose $b$ as $b = b_1b_2$ with $b_2$ in $A_S^+$ and $b_1$ a positive $t'$-ribbon-$t''$ for some $t''$ in $S$, so that the left orthogonal splitting of~$b^{-1}t'b$ is~$b_2^{-1}t''b_2$. By the above argument~$b^{-1}t'b$ lies in $A_{Y}$, so~$t''$ has to lie in $Y$ and $b_2$ has to lie in $A_Y^+$. But~$b$ is reduced-$Y$. Indeed, we assumed that $b$ is reduced-$X$ and~that $b\nsucceq t$ . This imposes $b_2 = 1$, $b = b_1$ and~$bt'' = t'b$ for some $t''$ in $Y$. Finally we already have $X'b = bX$. Since $t'$ does not belong to $X'$, It follows that $t''$ cannot lie in $X$. Thus $t'' = t$ and we are done.     

Conversely,  Assume~$bt\succeq b$, then~$b$ is a positive ribbon-$\{t\}$.  Since it is a positive ribbon-$X$, it is  also a positive ribbon-$Y$. Denote by $Y(t)$ the irreducible component of $Y$ that contains $t$. Since $t$ lies in $\partial(X)$, $Y(t)$ contains some element of $X$. By Proposition~\ref{lm58} (i)(b) if $t$ is a right-divisor of $b$ then so are all the element of $Y(t)$. But  ~$b$ is reduced-$X$. Thus $t$ does not right-divide $b$.
\end{proof}
Note that we showed the above result without using the assumption:~$b$ is a positive ribbon-$X^\bot$. This hypothesis is then useless for Lemma \ref{lm514}.
\begin{lm}\label{lm516}
Under the assumptions of proposition \ref{prp53}, we have $$Supp(b)=S.$$
\end{lm}
\begin{proof}
By assumption $b\neq 1$, so its support is not empty. Assume by contradiction that~$Supp(b)\neq S$. Let~$U$ be an indecomposable component of~$Supp(b)$. Fix~$u$ in $\partial U$ and set $V = Supp(b)\setminus U$. By hypothesis~$u$ does not lie in $Supp(b)$.  Then,~$u$ does not right-divide~$b$.  We claim that $bub^{-1}$ lies in $A^+_S$. Indeed,~$b$ is a positive ribbon-$X\cup X^\bot$  so if~$u$ belongs to $X\cup X^\bot$ there is nothing to say; if~$u$ lies in $\partial X$, then~$bu\succeq b$ by Lemma \ref{lm514}. Now, the set~$U$ is an indecomposable component of~$Supp(b)$, then each element of $U$ commute with each element of~$V$ and we can write $b = b_2b_1$ with~$b_1\in A_U^+$, and~$b_2\in A_{V}^+$. Write $b_1 = b'_1s$ with $s\in U$. Since~$U\cup\{u\}$ is indecomposable, there exists~$u_1,\dots,u_n\in U$ such that $u_0 = u$, $u_n =s$ and~$m_{u_iu_{i+1}}>2$.  Up to replacing $s$ by some $u_i$ with $i<n$, we can assume that $b$ has no right-divisor among~$u_1,\dots,u_{n-1}$.
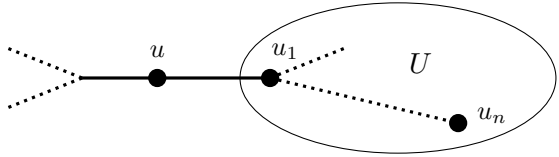
\begin{center}
\begin{tikzpicture}[decoration={brace}]
\draw[very thick,fill=black] (2,0) circle (.1cm);
\draw[very thick,fill=black] (3.5,0) circle (.1cm);
\draw[very thick,fill=black] (6,-.6) circle (.1cm);
\draw[very thick] (2,0) -- +(-1,0);
\draw[very thick] (2,0) -- +(1.5,0);
\draw[dotted,very thick] (3.5,0) -- +(1,.4);
\draw[dotted,very thick] (3.5,0) -- +(2.5,-.6);
\draw[dotted,very thick] (1,0) -- +(-1,.4);
\draw[dotted,very thick] (1,0) -- +(-1,-.4);
\draw (2,.35) node {$u$};
\draw (3.7,.35) node {$u_1$};
\draw (5.5,.2) node {{\large$U$}};
\draw (6.45,-.5) node {$u_n$};
\draw (5.2,0) ellipse (2.1 cm and 1cm);
\end{tikzpicture}
\end{center}
 Set~$U'=\{u,u_1,\dots,u_{n-1}\}$.  Let~$u_i$ lies in $U'$. If  $u_i$ does not lies in $X\cup X^\bot$, then~$u_i$ belongs to $\partial X$ and, by Lemma \ref{lm514},~$bu_i = v_i b$  for some $v_i$ in $S$. On the other hand~$b$ is a positive ribbon-$(X\cup X^\bot)$ therefore $b$ is also a positive ribbon-$U'$. The graph~$\Gamma_{U'}$ is connected by definition,~$u_n$ lies in $\partial U'$ and  right-divides $b$. Then by Proposition~\ref{lm58}, the positive elementary ribbon~$d_{U',u_n}$ right-divides $b$. Applying Lemma \ref{lm59}, we get that~$U'$ is contained in the support of $b$. Hence $u$ belongs to $Supp(b)$, a contradiction. So $Supp(b) = S$.
\end{proof}

We are now ready to prove Lemma \ref{lm517}
\begin{proof}[proof of Lemma \ref{lm517}]
Let~$s\in S\setminus X$, and set~$Y=S\setminus\{s\}$. Write~$b=b_1b_2$ with~$b_2\in A_Y^+$ and~$b_1$ reduced-$Y$.   By Lemma \ref{lm516}, we have~$Supp(b)=S$.  Since $b_2$ lies in $A_Y^+$, it follows that $b_1\neq 1$. In addition, $b_1$ is reduced-$Y$. Then, $s$ has to right-divide $b_1$.  We have~$b\D_Y^2b^{-1}=b_1\D_Y^2b_1^{-1}$. According to the assumptions of Proposition~\ref{prp53}, we have~$b\D_Y^2  = zb$ for some $z$ in $A^+_S$. Indeed, if~$\varepsilon(Y)=1$ then~$b\D_Y=z_1b$ for some~$z_1\in A_S^+$. Therefore~$b\D_Y^2=z_1b\D_Y=z_1^2b$. By Proposition~\ref{lm58}, we deduce that $b_1$ is a $Y'$-ribbon-$Y$ for some $Y'\subseteq S$ and $b = b_1b_2 = b'_2b_1$ with $b_2'\in A_{Y'}^+$. Since $s$ right-divides $b_1$, it has also to right-divide $b$.
\end{proof}
\subsection{When~$\Gamma_S$ is not connected}
In  Theorem~\ref{theointro1} we consider irreducible Artin-Tits group~$A_S$ of spherical type. Here we extend the theorem to any spherical type Artin-Tits group~$A_S$.

\begin{thm}\label{thm59}
Let~$A_S$ be an Artin-Tits group of spherical type. Denote the indecomposable components of~$S$  by~$S_1,\dots,S_n$.  Let~$A_X$ be a standard parabolic subgroup of~$A_S$ and set~$X_i=X\cap S_i$ for all~$i$. Set $$I=\{1\leq i\leq n\mid X_i\neq S_i, \D_{S_i}\in \DZ_{A_{S_i}}(A_{X_i})\textrm{ and }\D_{S_i}\notin Z(A_{S_i})\}.$$  $$J = \{1\leq i\leq n\mid X_i\neq S_i,\textrm{ and }  i\not\in I\}.$$ Finally, set $S_I = \bigcup_{i\in I} S_i$ and  $S_{J}  = \bigcup_{i \in J} S_i$. Then we have

$$\DZ_{A_S}(A_X)  = A_X\times QZ(A_{S_I})\times Z(A_{S_J}).$$

\end{thm}
\begin{proof}
For any direct product of groups~$G=G_1\times\cdots\times G_n$  and a subgroup $H$ of $G$ we have $$Z_G(H)=Z_{G_1}(H_1)\times\cdots\times Z_{G_n}(H_n).$$ where $H_i = H\cap G_i$ for each $i$.
Here, $A_S  = A_{S_1}\times\cdots\times A_{S_n}$ and $A_X \cap A_{S_i}  = A_{X_i}$.
Now by Theorem~\ref{theointro1}, if $i$ lies in $I$, then  $\DZ_{A_{S_i}}(A_{X_i}) = A_{X_i}\times QZ(A_{S_i})$;  if $i$ lies in $J$ then $\DZ_{A_{S_i}}(A_{X_i}) = A_{X_i}\times Z(A_{S_i})$. In addition, if $i$ is neither in $I$ nor in $J$, then $X_i = S_i$ and $\DZ_{A_{S_i}}(A_{X_i}) = A_{X_i}$. So, we deduce that
$$\DZ_{A_S}(A_X)  = Z_{A_S}(\prod_{i = 1}^nZ_{A_{S_i}}(A_{X_i})) = \prod_{i = 1}^n \DZ_{A_{S_i}}(A_{X_i})    =$$ $$\prod_{i\in I} (A_{X_i}\times QZ(A_{S_i})) \times\prod_{i\in J} (A_{X_i}\times Z(A_{S_i}))\times \prod_{i\not\in I\cup J} A_{X_i} = $$ $$ \prod_{i = 1}^n A_{X_i}\times \prod_{i\in I} QZ(A_{S_i})\times\prod_{i\in J}  Z(A_{S_i}).$$ But $\prod_{i = 1}^n A_{X_i} = A_X$, $\prod_{i\in I} QZ(A_{S_i}) = QZ(A_{S_I})$ and $\prod_{i\in J} QZ(A_{S_i}) = QZ(A_{S_J})$. So the equality holds.
 \end{proof}

\subsection{Application the subgroup conjugacy problem}
Given a group~$G$ and a subgroup~$H$ of~$G$, the subgroup conjugacy problem for~$H$ is solved by finding an algorithm that determines whether any two given elements of~$G$ are conjugated by an element of~$H$. In this section, we focus on Artin-Tits groups of type~$B$ or~$D$ and use Theorem~\ref{theointro1} and \cite[Theorem 1.1]{Par2} to reduce the subgroup conjugacy problem for their irreducible standard parabolic subgroups to an instance of the simultaneous conjugacy problem. We follow the strategy used in \cite{GKLT} to solve  the subgroup conjugacy problem for irreducible standard parabolic subgroups of an Artin-Tits group of type $A$. The simultaneous conjugacy problem is solved  for Artin-Tits groups of type~$A$ in \cite{Men3} (see also  \cite{LeeLee}), but the result and its proof can be generalized verbatim to all Artin-Tits groups of spherical type, in particular to Artin-Tits groups of type $B$ or type~$D$ . Hence, we obtain a solution to the subgroup conjugacy problem for irreducible standard parabolic subgroups of Artin-Tits groups of type $B$ and $D$.\\

Let us recall \cite[Theorem 1.1]{Par2} and \cite[Theorem 2.13]{GKLT}.
\begin{thm}[\cite{Par2}, Theorem 1.1]\label{thm60}
Let~$A_S$ be an Artin-Tits group of spherical type such that~$\Gamma_S=A_k~(k\geq1)$,~$\Gamma_S=B_k~(k\geq2)$ or~$\Gamma_S=D_k~(k\geq4)$. Let~$X\subseteq S$ such that~$\Gamma_X$ is connected. Then~$Z_{A_S}(A_X)$ is generated by $$X^\perp\cup\{\D_Y\in Z_{A_S}(A_X) \mid X\subseteq Y\}\cup\{\D_Y\D_{Y'}\in Z_{A_S}(A_X) \mid X\subseteq Y,X\subseteq Y' \}.$$
\end{thm}
Note that in the third set, we can restrict the pair $(Y,Y)$ to those so that neither $\D_Y$ nor $\D_{Y'}$ belong to $Z_{A_S}(A_X)$. In the sequel, we denote the obtained generating set  by $\Upsilon(X)$. 
\begin{ex}\label{ex61}
Consider $S = \{s_1,s_2,s_3\}$ with~$A_S$ of type~$B_3$ as below. Set $X = \{s_2\}$. 
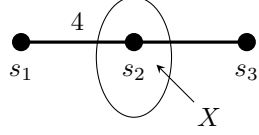
\begin{center}
\begin{tikzpicture}
\draw[very thick,fill=black] (2,-1.5) circle (.1cm);
\draw[very thick,fill=black] (3.5,-1.5) circle (.1cm);
\draw[very thick,fill=black] (5,-1.5) circle (.1cm);
\draw[very thick] (2,-1.5) -- +(3,0);
\draw (2,-1.9) node {$s_1$};
\draw (3.5,-1.9) node {$s_2$};
\draw (5,-1.9) node {$s_3$};
\draw (2.75,-1.23) node {$4$};
\draw (4.5,-2.5) node {$X$};
\draw (3.5,-1.7) ellipse (.5cm and .8cm);
\draw[-stealth] (4.3,-2.3) -- +(-.5,.5);
\end{tikzpicture}
\end{center}~
We have $X^\bot=\emptyset$ and  $Z_{A_S}(A_X)$ is generated by $\Upsilon(X) = \{ s_2, \Delta_{\{s_1,s_2\}}, \Delta_S, \Delta^2_{s_2,s_3}\}$. 
\end{ex}
\begin{thm}[\cite{GKLT}, Theorem 2.13]\label{thm61}
Let~$G$ be a group and~$H$ be a subgroup such that~$\DZ_G(H)=Z(G)\cdot H$. Suppose further that~$Z_G(H)$ is generated by a set~$\{g_1,\dots,g_n\}$. Then for~$x,y\in G$, the following are equivalent:
\begin{enumerate}
\item there exists~$c\in H$ such that~$y=c^{-1}xc$.
\item there exists~$z\in G$ such that
\begin{enumerate}
\item~$y=z^{-1}xz$, and
\item[($b_i$)]~$g_i=z^{-1}g_iz$ for all~$1\leq i\leq n$.
\end{enumerate}
\end{enumerate}
\end{thm}

\begin{cor}\label{cor51}
Let~$A_S$ be an Artin-Tits group of type~$B_k~(k\geq2)$ or~$D_k~(k\geq4)$. Let~$X\subseteq S$ be such that~$\Gamma_X$ is connected. In case~$\Gamma_S$ is of type~$D_{2k+1}$,  assume~$\{s_2,s_{2'},s_3\}$ is not included in $X$ with the notations of Figure \ref{figuretypeD}. For any pair~$(x,y)$ of elements of~$A_S$, the following are equivalents:
\begin{enumerate}
\item there exists~$c\in A_X$ such that~$y=c^{-1}xc$.
\item there exists~$z\in A_S$ such that
\begin{enumerate}
\item $y=z^{-1}xz$,
\item $g = z^{-1}g z$ for all~$g$ in $\Upsilon(X)$.
\end{enumerate}
\end{enumerate}
\end{cor}
\begin{proof} By Theorem~\ref{theointro1} and Proposition~\ref{condtech} we have~$\DZ_G(A_X))=Z(A_S)\times A_X$.  So we are in position to apply  Theorem~\ref{thm61}. 
\end{proof}
\section{The non spherical type cases}
We turn now to the proof of Theorem~\ref{theointro2} that is concerned with Artin-Tits groups that are not of spherical type. Our main argument is Proposition~\ref{thmlastsect}. Indeed, In~\cite{God4} the second author stated several conjectures, that are proved to hold for Artin-Tits groups of various types. Our proof is based on these conjectures.  
\begin{prp}\label{thmlastsect}
\begin{enumerate}
\item Let  $A_S$ be an Artin-Tits group. Assume $A_S$ has the property $(\circledast)$ stated in~\cite{God4}, then for any $X$  included in  $S$ one has 
\begin{enumerate} 
\item If~$A_X$ is of spherical type, then for any positive integer $k$, $$Z_{A_S}(\Delta^{2k}_X)= N_{A_S}(A_X).$$
\item if~$A_X$ is of spherical type and there is no $Y$ of spherical type and containing $X$, then  $$QZ_{A_S}(A_X)  =  QZ(A_X) \textrm{ and } N_{A_S}(A_X) = A_{X}$$
\item  If~$A_X$ is irreducible and not of spherical type, then $$QZ_{A_S}(A_X) = A_{X^\perp} \textrm{ and } N_{A_S}(A_X) = A_{X\cup X^\perp}$$
\end{enumerate}
\item\cite{God,God_pjm,God4} If $A_S$ is of spherical type, of FC type, of large type or of 2-dimensional type. Then, $A_S$ has the Property~$(\circledast)$. 
\end{enumerate}
\end{prp}

\begin{proof} (i) Conjecture~$(\circledast)$ implies that $N_{A_S}(A_X) = A_X\cdot QZ_{A_S}(A_X)$ and that $QZ_{A_S}(A_X)$ is the subgroup of $A_S$ generated by the set of positive $X$-ribbons-$X$ (see~\cite{God4}). If~$A_X$ is irreducible and not of spherical type, then the set of elementary positive $X$-ribbons is equal to $X^\perp$. Moreover all the elements of  $X^\perp$ are $X$-ribbons-$X$. So  $QZ_{A_S}(A_X) = A_{X^\perp}$ and Point~(c) holds.  Assume $A_X$ is of spherical type.  Fix a positive integer $k$. If $g$ lies in $Z_{A_S}(\Delta^{2k}_X)$, then in particular $g^{-1}\Delta^{2k}_Xg$  belongs to $A_X$.  Property~$(\circledast)$ imposes that that $g$ belongs to $A_X\cdot QZ_{A_S}(A_X)$, that is to $N_{A_S}(A_X)$. Conversely,  $A_X\cdot QZ_{A_S}(A_X)$ is included in $Z_{A_S}(\Delta^{2k}_X)$ because both $A_X$ and $QZ_{A_S}(A_X)$ have to fix the center of $A_X$, which contains $\Delta^2$. So Point~(a) holds. Finally, if there is no $Y$ of spherical type and containing $X$, then the elementary positive ribbons $d_{X,t}$ are the elements~$\Delta_{X(t)}$ with $t$ in $X$ (see Definition~\ref{defribb}). It follows that $QZ_{A_S}(A_X)$ is included in $A_X$ and is, therefore, equal to  $QZ(A_X)$. Since $N_{A_S}(A_X) = A_X\cdot QZ_{A_S}(A_X)$, we deduce that $N_{A_S}(A_X) = A_X$. Hence Point~(b) holds.
\end{proof}
In the sequel we first extend Conjecture~\ref{conjintro} to the context of non irreducible parabolic subgroup  (see Conjecture~\ref{conjintrogener}). Then we prove that Conjecture~\ref{conjintrogener} holds for  any Artin-Tits which possesses the property $(\circledast)$  (see Theorem~\ref{propfin2}).  Considering 
Proposition~\ref{thmlastsect} (ii), this will prove  Theorem~\ref{theointro2}.
\begin{conj}\label{conjintrogener}
Let  $A_S$ be an irreducible Artin-Tits group and  $X$ be included in  $S$.  Let $X_s$ be the union of the irreducible components of $X$ that are of spherical type, and $X_{as}$ be the union of the other irreducible components of $X$. 
Then, $$\DZ_{A_S}(A_X) = Z_{A_S}(Z_{A_{X^\perp_{as}}}(A_{X_s}))$$  
\begin{enumerate}
\item Assume $X_s$ is empty. Then $$\DZ_{A_S}(A_X)= Z_{A_S}(A_{X^\perp}).$$
\item Assume $A_X$ is of spherical type. 
Let $A_T$ be the smallest standard parabolic subgroup of $A_S$ that contains $Z_{A_S}(A_{X})$.  
\begin{enumerate} 
\item If $T$ is of spherical type then  $$\DZ_{A_S}(A_X) = \DZ_{A_T}(A_{X}).$$
\item If $T$ is not of spherical type then  $$\DZ_{A_S}(A_X)=A_{X}.$$
\end{enumerate}
\end{enumerate}
\end{conj} 

\begin{prp} \label{propfin1} Let  $A_S$ be an irreducible Artin-Tits group and  $X$ be included in  $S$. Assume $A_S$ has the  property~$(\circledast)$ stated in~\cite{God4}. Conjecture~\ref{conjintrogener} implies Conjecture~\ref{conjintro}. 
\end{prp}  
\begin{proof}
Consider the notations of Conjecture~\ref{conjintro}. Assume $X$ is irreducible. If~$X$ is not of spherical type, then $X = X_{as}$ and $X_s$ is empty. By Proposition~\ref{thmlastsect},  $Z_{A_S}(A_X) \subseteq QZ_{A_S}(A_X) = A_{X^\perp}\subseteq Z_{A_S}(A_X)$. Therefore  $A_{X^\perp}= Z_{A_S}(A_X)$ and $T = X^\perp$.  Thus, Conjecture~\ref{conjintrogener}(i) implies Conjecture~\ref{conjintro}(i).  In the case~$X$ is of spherical type, there is nothing to prove. \end{proof}

\begin{thm} \label{propfin2}  Let  $A_S$ be an irreducible Artin-Tits group. If $A_S$ has the property $(\circledast)$ stated in~\cite{God4}, then Conjecture~\ref{conjintrogener} holds. 
\end{thm}

In order to prove Theorem~\ref{propfin2} We need some preliminary results. In the sequel,  we assume  $A_S$ is an irreducible Artin-Tits group that has the property $(\circledast)$ stated in~\cite{God4}. We fix a standard parabolic subgroup $A_X$ with $X\subseteq S$. By $X_s$ we denote the union of the irreducible components of $X$ that are of spherical type. By $X_{as}$ we denote the union of the other irreducible components of $X$.  By definition $X_s$ is included in $X_{as}^\perp$. We set $$\Upsilon = \{Y\subseteq S\mid X_s\subseteq Y;\textrm{  and }A_Y\textrm{ is of spherical type.}\}$$ Let $A_T$ be the smallest standard parabolic subgroup of $A_S$ that contains $Z_{A_{S}}(A_{X_s})$.

\begin{lm}\label{lemfin1} $Z_{A_S}(A_X) = Z_{A_{X_{as}^\perp}}(A_{X_s})$.
\end{lm}

\begin{proof} We have $A_X = A_{X_s}\times A_{X_{as}}$. Therefore $Z_{A_S}(A_X) = Z_{A_S}(A_{X_s}) \cap Z_{A_S}(A_{X_{as}})$.  Let $X_1,\cdots, X_k$ be the irreducible components of $X_{\textcolor{red}{s}}$. Then $X_{as}^\perp = X^\perp_{1}\cap\cdots\cap X^\perp_{k}$. On the other hand, $A_{X_{as}} = A_{X_1}\times \cdots \times A_{X_k}$ and $Z_{A_S}(A_{X_{as}}) =  Z_{A_S}(A_{X_1})\cap\cdots \cap  Z_{A_S}(A_{X_k})$. By Proposition~\ref{thmlastsect}, $Z_{A_S}(A_{X_{i}}) = QZ_{A_S}(A_{X_{i}}) = A_{X^\perp_{i}}$ for each component $X_i$. Therefore 
 $Z_{A_S}(A_{X_{as}}) = A_{X^\perp_{1}}\cap\cdots\cap A_{X^\perp_{k}} = A_{X^\perp_{1}\cap\cdots\cap X^\perp_{k}} =A_{X_{as}^\perp}$.   But, $A_{X_s}$ is included in $ A_{X^\perp_{as}}$. Thus, $Z_{A_S}(A_{X_s}) \cap Z_{A_S}(A_{X_{as}}) = Z_{A_{X_{as}^\perp}}(A_{X_s})$.
\end{proof}

\begin{lm} \label{lemfin2} The set $\Upsilon$ is not empty and all its elements are contained in $T$. Moreover,~$T$ belongs to $\Upsilon$ if and only if $T$ is of spherical type. In this case, $T$ is the unique maximal element of $\Upsilon$.
\end{lm}
\begin{proof} $X_s$ is contained in $\Upsilon$, so the latter is not empty. Moreover, $X_s$ is included in $T$. Therefore the latter  belong to $\Upsilon$ if and only if it is of spherical type. Finally if $Y$ belongs  to $\Upsilon$, then $\Delta_Y^2$ belongs to $Z_{A_S}(A_Y)$, and therefore to $Z_{A_S}(A_X)$.  Thus, $Y$ is included in $T$. Hence, if $T$ belongs to $\Upsilon$, it is its unique maximal element.  
\end{proof}

\begin{lm}\label{lemfin3}  Assume $Y$ is maximal in $\Upsilon$ for the inclusion. Then, $$\DZ_{A_S}(A_{X_s}) \subseteq \DZ_{A_Y}(A_{X_s})$$
\end{lm}
\begin{proof} Assume $g$ belongs to $\DZ_{A_S}(A_{X_s})$. The element $\Delta_Y^2$  lies in $Z(A_Y)$. Since $X_s$ is included in~$Y$, it follows that $\Delta_Y^2$ lies in $Z_{A_S}(A_{X_s})$, and $g\Delta_Y^2g^{-1} = \Delta_Y^2$. By Proposition~\ref{thmlastsect}(i)(a),  $g$ belongs to the subgroup $N_{A_S}(A_Y)$. But $Y$ is maximal in $\Upsilon$. By Proposition~\ref{thmlastsect}(i)(b), $N_{A_S}(A_Y) = A_Y$. Thus $DZ_{A_S}(A_{X_s}) = Z_{A_S}(Z_{A_S}(A_{X_s}))\cap A_Y =  Z_{A_Y}(Z_{A_S}(A_{X_s})) \subseteq Z_{A_Y}(Z_{A_Y}(A_{X_s})) = \DZ_{A_Y}(A_{X_s})$.
\end{proof}

We can now prove Theorem~\ref{propfin2} 
\begin{proof}[Proof of Theorem~\ref{propfin2}.]
By Lemma~\ref{lemfin1}, we have $Z_{A_S}(A_X) = Z_{A_{X_{as}^\perp}}(A_{X_s})$. It follows that~$\DZ_{A_S}(A_X) = Z_{A_S}(Z_{A_{X_{as}^\perp}}(A_{X_s}))$. When $X_s$, is empty, we have $X_{as} = X$, $A_{X_s} = \{1\}$. So  $Z_{A_{X_{as}^\perp}}(A_{X_s}) = Z_{A_{X^\perp}}(\{1\}) = A_{X^\perp}$. Therefore, ~$\DZ_{A_S}(A_X) = Z_{A_S}(A_{X^\perp})$. Assume for the remaining of the proof that  $X$ is of spherical type. Assume, First, that $T$ is of spherical type. By Lemma~\ref{lemfin2}, $T$ is maximal in $\Upsilon$ and, by Lemma~\ref{lemfin3},  $\DZ_{A_S}(A_X) \subseteq \DZ_{A_T}(A_{X})$.  On the other hand,  $Z_{A_T}(A_{X}) = Z_{A_S}(A_{X})\cap A_T = Z_{A_S}(A_X)$. We deduce that $\DZ_{A_T}(A_{X}) = Z_{A_T}(Z_{A_S}(A_X)) \subseteq \DZ_{A_S}(A_X)$.  Hence,  $\DZ_{A_S}(A_X) = \DZ_{A_T}(A_{X})$. Assume, finally, that $T$ do not lie in~$\Upsilon$. Let $Y$ be maximal in $\Upsilon$.  By Lemma~\ref{lemfin3},  we get $\DZ_{A_S}(A_X) \subseteq \DZ_{A_Y}(A_{X})$. If $Y = X$, then  $A_{X} \subseteq \DZ_{A_S}(A_{X}) \subseteq \DZ_{A_{X}}(A_{X}) = A_{X}$ and we are done. So, assume $X\subsetneq Y$. The group $A_Y$ is of spherical type. Applying Theorem~\ref{theointro1}, we get that  $\DZ_{A_Y}(A_{X}) \subseteq QZ(A_Y)\times A_{X}$.  Since $A_{X}$ is included in $\DZ_{A_S}(A_{X})$, the group~$A_{X}$ is equal to $\DZ_{A_S}(A_{X})$ if and only if $\DZ_{A_S}(A_{X})\cap QZ(A_Y) = \{1\}$. Assume this is not the case. Then, there exists $k > 0$  so that $\Delta^k_Y$ lies in $\DZ_{A_S}(A_{X})$. We can assume without restriction that $k$ is even. Since $Y$ lies in $\Upsilon$ and $T$ does not, they are distinct. It follows from the definition of $T$ that there exists $g$ in $Z_{A_S}(A_{X})$ which is not in $A_Y$. But  $\Delta^k_Y$ lies in $\DZ_{A_S}(A_{X})$. So  we have $\Delta^k_Yg (\Delta^k_Y)^{-1} = g$, and equivalently $g \Delta^k_Yg^{-1} = \Delta^k_Y$.  The latter equality imposes that $g$ belongs to
$N_{A_S}(A_Y)$ by Proposition~\ref{thmlastsect}(i)(a). But $N_{A_S}(A_Y) = A_Y$ by Proposition~\ref{thmlastsect}(i)(b), a contradiction. Hence, $\DZ_{A_S}(A_{X}) = A_{X}$. 
\end{proof}

\begin{cor} \label{corsecfinale} Let  $A_S$ be an irreducible Artin-Tits group of FC type, or of large type, or of $2$-dimensional type. Then Conjecture~\ref{conjintrogener} holds. 
\end{cor}

\begin{rmq} In an (irreducible) Artin-Tits group that is of large type, all standard parabolic subgroups are irreducible.  So, Corollary~\ref{corsecfinale} provides a complete description of the double centralizer of any standard parabolic subgroup. However, for the other not spherical types in the case both $X_s$  and $X_{as}$ are not empty, the answer is not completely satisfactory.  Indeed the double centralizer is not as simple as in the cases where either $X_s$ or $X_{as}$ is empty. For instance, in the following example, we have $Z_{A_S}(A_X) = Z(A_{X_s})$ and $DZ_{A_X} = N_{A_S}(A_{X_s}) = QZ_{A_S}(A_{X_s}) \cdot  A_{X_s}$
\end{rmq}
\begin{center}
\begin{figure}[!h]
\begin{tikzpicture}[scale=0.75]
\draw[very thick,fill=black] (3.5,-13.1) circle (.1cm);
\draw[very thick,fill=black] (5,-13.1) circle (.1cm);
\draw[very thick,fill=black] (5,-11.6) circle (.1cm);
\draw[very thick,fill=black] (6.5,-13.1) circle (.1cm);
\draw[very thick,fill=black] (8,-13.1) circle (.1cm);
\draw[very thick,fill=black] (9.5,-13.1) circle (.1cm);
\draw[very thick,fill=black] (11,-13.1) circle (.1cm);
\draw[very thick] (3.5,-13.1) -- +(7.5,0);
\draw[very thick] (3.5,-13.1) -- +(1.5,1.5);
\draw[very thick] (5,-11.6) -- +(0,-1.5);
\draw (3.5,-13.5) node {$s_2$};
\draw (5,-13.5) node {$s_3$};
\draw (6.5,-12.7) node {$s_4$};
\draw (8,-13.5) node {$s_5$};
\draw (9.5,-13.5) node {$s_6$};
\draw (11,-13.5) node {$s_7$};
\draw (5,-11.2) node {$s_1$};
\draw (11.8,-12.2) node {$X_s$};
\draw (9.5,-13.3) ellipse (2cm and .8cm);
\draw[-stealth] (11.4,-12.4) -- +(-.45,-.25);
\draw (2.2,-10.7) node {$X_{as}$};
\draw (4.25,-12.7) ellipse (1.25cm and 2.5cm);
\draw[-stealth] (2.7,-10.9) -- +(.45,-.25);

\end{tikzpicture}
\end{figure}
\end{center}

\end{document}